\documentclass[a4paper]{amsart}
\usepackage{hyperref, aliascnt}

\DeclareMathOperator{\GL}{GL}
\DeclareMathOperator{\slMath}{sl}

\newcommand{\ca}{$C^*$-algebra}
\newenvironment{psmallmatrix}
  {\left(\begin{smallmatrix}}
  {\end{smallmatrix}\right)}

\newtheorem{lma}{Lemma}[section]

\newaliascnt{thmCt}{lma}
\newtheorem{thm}[thmCt]{Theorem}
\aliascntresetthe{thmCt}

\newaliascnt{corCt}{lma}
\newtheorem{cor}[corCt]{Corollary}
\aliascntresetthe{corCt}

\newaliascnt{prpCt}{lma}
\newtheorem{prp}[prpCt]{Proposition}
\aliascntresetthe{prpCt}

\theoremstyle{definition}

\newaliascnt{rmkCt}{lma}
\newtheorem{rmk}[rmkCt]{Remark}
\aliascntresetthe{rmkCt}

\newaliascnt{exaCt}{lma}
\newtheorem{exa}[exaCt]{Example}
\aliascntresetthe{exaCt}

\newaliascnt{qstCt}{lma}
\newtheorem{qst}[qstCt]{Question}
\aliascntresetthe{qstCt}

\newcounter{theoremintro}

\newtheorem{exaIntro}[theoremintro]{Example}
\newaliascnt{thmIntroCt}{theoremintro}
\newtheorem{thmIntro}[thmIntroCt]{Theorem}
\aliascntresetthe{thmIntroCt}

\def\today{\number\day\space\ifcase\month\or   January\or February\or
   March\or April\or May\or June\or   July\or August\or September\or
   October\or November\or December\fi\   \number\year}

\title{Products of commutators in matrix rings}
\date{\today}

\author{Matej Bre\v sar}
\address{Matej Bre\v sar, Faculty of Mathematics and Physics,  University of Ljubljana; 
Faculty of Natural Sciences and Mathematics, University 
of Maribor; and Institute of Mathematics, Physics, and Mechanics, Ljubljana, Slovenia} \email{matej.bresar@fmf.uni-lj.si}

\urladdr{www.fmf.uni-lj.si/en/directory/21/bresar-matej}

\author[Eusebio Gardella]{Eusebio Gardella}
\address{Eusebio Gardella
Department of Mathematical Sciences, Chalmers University of
Technology and University of Gothenburg, Gothenburg SE-412 96, Sweden.}
\email{gardella@chalmers.se}
\urladdr{www.math.chalmers.se/~gardella}

\author{Hannes Thiel}
\address{Hannes Thiel,
Department of Mathematical Sciences, Chalmers University of
Technology and University of Gothenburg, Gothenburg SE-412 96, Sweden.}
\email{hannes.thiel@chalmers.se}
\urladdr{www.hannesthiel.org}

\thanks{The first named author was partially supported by the Slovenian Research Agency (ARRS) Grant P1-0288.
The second named author was partially supported by the Swedish Research Council Grant 2021-04561.
The third named author was partially supported by the Knut and Alice Wallenberg Foundation (KAW 2021.0140).
}

\subjclass[2020]%
{Primary 
16S50. 
Secondary 
15A30, 
16K40, 
16W25, 
46L05, 
46L10. 
}
\date{\today}

\begin{document}

\begin{abstract} Let $R$ be a ring and let $n\ge 2$.
We discuss the question of whether every element in the matrix ring $M_n(R)$ is a product of (additive) commutators $[x,y]=xy-yx$, for $x,y\in M_n(R)$. 
An example showing that this does not always hold, even when $R$ is commutative, is provided. 
If, however, $R$ has Bass stable rank one, then under various additional conditions every element in $M_n(R)$ is a product of three commutators. 
Further, if $R$ is a division ring with infinite center, then every element in $M_n(R)$ is a product of two commutators. 
If $R$ is a field and $a\in M_n(R)$, then every element in $M_n(R)$ is a sum of elements of the form $[a,x][a,y]$ with $x,y\in M_n(R)$ if and only if the degree of the minimal polynomial of $a$ is greater than $2$.
\end{abstract}

\maketitle

\section{Introduction}

By the commutator of elements $a$ and $b$ in a ring we will always mean the additive commutator $[a,b]=ab-ba$.
The second and third named authors recently showed that if a unital ring $S$ is generated by its commutators as an ideal, then there exists a natural number $N$ such that every element $a \in S$ is a sum of $N$ products of pairs of commutators, that is, $a=\sum_{i=1}^N [b_i,c_i][d_i,e_i]$ for some $b_i,c_i,d_i,e_i\in S$;
see \cite[Theorem~3.4]{GarThi23arX:GenByCommutators}.
The minimal such $N$, denoted  $\xi(S)$, was computed or estimated for various classes of rings and \ca{s}. 
In particular, for any unital, possibly noncommutative ring $R$, the ring
$M_n(R)$  of $n$-by-$n$ matrices over~$R$ satisfies $\xi(M_n(R)) \le 2$ for every $n\ge 2$;
see \cite[Theorem~5.4]{GarThi23arX:GenByCommutators}.

This paper is mainly concerned with the question of whether every
matrix in~$M_n(R)$ is actually the product of (two or more) commutators rather than a sum of double products. 
The fundamental case where $R=F$ is a field was treated quite a while ago by Botha who proved that every matrix in $M_n(F)$ is a product of two commutators \cite[Theorem~4.1]{Bot97ProdMatPrescribedTraces}, that is to say, $\xi(M_n(F))=1$ for every field~$F$ and every $n\ge 2$ (for fields of characteristic $0$ this was proved earlier in \cite{Wu89OpFactorization}). 
We will be interested in more general rings.

Our problem can be placed in a more general context. 
Over the last years, there has been a growing interest in images of noncommutative polynomials in matrix algebras. 
We refer the reader to the recent survey \cite{KanMalRowYav20EvalNcPolynomials} on this topic. 
Note that the condition that $\xi(M_n(R)) =1$ can be reformulated as  saying that the image of the polynomial $f= [X_1,X_2][X_3,X_4]$  on $M_n(R)$ is the whole $M_n(R)$, and the aforementioned result by Botha confirms the 
L'vov-Kaplansky conjecture for $f$;  
this conjecture states that the image of any multilinear polynomial on $M_n(F)$ is a vector subspace.

\medskip

Let us  present the main results of this paper.
In \autoref{sec:NotProduct}, we provide an example showing the nontriviality of our problem. 
The following is a simplified version of \autoref{prp:NotProdComm}.

\begin{exaIntro}
\label{eA}
There exist a commutative, unital ring $R$ and a matrix $a \in M_2(R)$ that  cannot be written as a product of commutators. 
\end{exaIntro}

Together with the aforementioned result from \cite{GarThi23arX:GenByCommutators}, Example \ref{eA} shows that there exist rings $R$ such that $\xi(M_2(R))= 2$. 
This answers \cite[Question~5.7]{GarThi23arX:GenByCommutators}.

\autoref{sec:Bsr} is primarily devoted to matrix algebras over algebras having Bass stable rank one. 
The following is a combination of \autoref{prp:ProdCommSR1} and \autoref{prp:ProdCommHermiteSR1}.

\begin{thmIntro}
\label{thmb}
Let $A$ be a unital algebra over an infinite field, and assume that $A$ has Bass stable rank one. 
Let $n\geq 3$.
Then the following statements hold:
\begin{enumerate} 
\item 
Every matrix in $\GL_n(A)$ is a product of three commutators.
\item 
If $A$ is right K-Hermite, then every matrix in $M_n(A)$ is a product of three commutators.
\end{enumerate}
\end{thmIntro}

We also prove that if $A$ is any unital algebra over an infinite field and $n\geq 3$, then
every triangular matrix in $M_n(A)$ is a product of two commutators (\autoref{prp:TriangularProductCommutators}). 
This is needed in the proof of \autoref{thmb}, but is  of independent interest. 

Matrix rings over division rings are the topic of \autoref{sec:DivRgs}.
The following is \autoref{prp:ProdCommInftyCenter}.

\begin{thmIntro}
Let $D$ be a division ring with infinite center and 
let $n\geq 2$.
Then every matrix $a\in M_n(D)$ is a product of two commutators. 
\end{thmIntro}
 
The assumption that the center is infinite  is unnecessary if either $n=2$ (\autoref{prp:ProdCommutatorsSize2}) or if $a$ is singular (\autoref{prp:SingularMatrix}). 
Its necessity in general is left open.

The final \autoref{sec:FixedElement} studies a variation of the problem from the preceding sections: 
We consider only commutators with a fixed element (that is, values of an inner derivation), but allow for sums of their products rather than only products. 
The following is \autoref{tder}.

\begin{thmIntro}
Let $F$ be a field, let $n\ge 2$, and let $a\in M_n(F)$. 
Then every matrix in $M_n(F)$ is   a sum of matrices of the form $[a,x][a,y]$ with $x,y\in M_n(F)$ if and only if the degree of the minimal polynomial of $a$ is greater than $2$.
\end{thmIntro}

\section{Matrices that are not products of commutators}
\label{sec:NotProduct}

In this section, we exhibit an example of a commutative, unital ring such that not every $2$-by-$2$ matrix over this ring is a product of (finitely many) commutators;
see \autoref{prp:NotProdComm}.

\medskip

Let $C$ be a commutative algebra over a field $F$.
We denote by $\slMath_2(C) \subseteq M_2(C)$ the space of trace zero matrices. 
Note that the commutator of any two matrices from $M_2(C)$ belongs to $\slMath_2(C)$.
For the case $C=F$, Albert and Muckenhoupt \cite{AlbMuc57MatricesTraceZero} (extending earlier work of Shoda \cite[Satz~3]{Sho37SaetzeMatrizen} in characteristic $0$) showed that the converse also holds, that is, a matrix over a field is a commutator if and only if it has trace zero.

In the following result, we consider the case where $C=F$.

\begin{lma}
\label{le}
Let $s_1,s_2,\dots,s_n\in \slMath_2(F)$ satisfy
\[
s_1s_2\cdots s_n=0. 
\]
Let $t_{1k},t_{2k},t_{3k},t_{4k}\in \slMath_2(F)$, for $k=1,\dots,n$, be any trace zero matrices, and set
\begin{align*}
\label{f1}    
r_1&=t_{11}s_{2}\cdots s_{n} + s_{1}t_{12}s_3\cdots s_{n} + \dots + s_{1}\cdots s_{n-1}t_{1n},\\
r_2&=t_{21}s_{2}\cdots s_{n} + s_{1}t_{22}s_3\cdots s_{n} + \dots + s_{1}\cdots s_{n-1}t_{2n},\\
 r_3&=t_{31}s_{2}\cdots s_{n} + s_{1}t_{32}s_3\cdots s_{n} + \dots + s_{1}\cdots s_{n-1}t_{3n},\\   r_4&=t_{41}s_{2}\cdots s_{n} + s_{1}t_{42}s_3\cdots s_{n} + \dots + s_{1}\cdots s_{n-1}t_{4n}.
\end{align*}
Then $\{r_1, r_2, r_3, r_4\}\subseteq M_2(F)$ is a linearly dependent set over $F$.\end{lma}
\begin{proof}
Set $\mathcal{R}=\{r_1, r_2, r_3, r_4\}$.
The proof is by induction on $n$. If $n=1$, then $r_i=t_{i1}$ for all $i=1,\dots,4$. 
Thus $\mathcal{R} \subseteq \slMath_2(F)$ must be linearly dependent since $\dim_F(\slMath_2(F))=3$. 

We may thus assume that the lemma is true for all positive integers less than~$n$. 
If $s_1$ is invertible, then $s_2\cdots s_n=0$
and hence the induction hypothesis implies that $s_1^{-1}\mathcal{R}$ is linearly dependent, so $\mathcal{R}$ is linearly dependent too. We may therefore assume that $s_1$ is not invertible, and, analogously, we may  assume that $s_n$ is not invertible. 

Being 2$\times$2 matrices with trace zero with zero determinant, $s_1$ and $s_n$ have square zero, which implies that $s_1r_is_n= 0$ for all $i=1,\dots, 4$. 
If $\mathcal{R}$ was linearly independent, then it would follow that $s_1M_2(F)s_n= \{0\}$, which is possible only if $s_1=0$ or $s_n=0$. 
Assume that $s_1=0$. Then $r_k=t_{k1}s_2\cdots s_n$ for $k=1,2,3,4$.
Set $x=s_2\cdots s_n$. Then $r_1,r_2,r_3,r_4\in \mathrm{sl}_2(F)x$,
which is at most three-dimensional. Thus $\mathcal{R}$ is linearly
dependent, which is a contradiction. 
The case $s_n=0$ is analogous, and in either case we deduce that
$\mathcal{R}$ is linearly dependent.
\end{proof}

Given a nonunital $F$-algebra $B$, recall that its (minimal) unitization is the $F$-algebra $C=B\oplus F$ with product given by $(a,\lambda)(b,\mu)=(\mu a + \lambda b + ab, \lambda\mu)$ for all $a,b\in B$ and all $\lambda,\mu\in F$. 

\begin{thm}
\label{prp:NotProdComm}
Let $C_0$ be a 4-dimensional $F$-algebra with zero multiplication, and let $\{c_1,c_2,c_3,c_4\}$ be a basis for $C_0$. 
Let $C$ be the unitization of $C_0$. 
Then the matrix
$a=\left[\begin{smallmatrix} c_1 & c_2\cr c_3 & c_4
\cr
\end{smallmatrix} \right]\in M_2(C)$
cannot be written as a product of elements in $\slMath_2(C)$. 
In particular, $a$ cannot be written as a product of commutators in $M_2(C)$. 
\end{thm}
\begin{proof}
Arguing by contradiction, suppose that there exist $t_1,t_2,\ldots, t_n\in \slMath_2(C)$ such that
\[
a= t_1t_2\cdots t_n.
\]
For each $k=1,\ldots,n$, there are $t_{1k},\ldots, t_{4k}\in M_2(F)$ and $t_{0k} \in F$ such that
\[
t_k=t_{0k}+c_1t_{1k} + c_2t_{2k}+c_3t_{3k}+c_4t_{4k}.
\]
Applying the trace $\tau$ of $M_2(C)$ to the identity above, and using
that $t_k\in \slMath_2(C)$ yields the identity
\[0=t_{0k}+c_1\tau(t_{1k}) + c_2\tau(t_{2k})+c_3\tau(t_{3k})+c_4\tau(t_{4k})\]
in $C$. Since $\{1,c_1,c_2,c_3,c_4\}$ is a linearly independent set 
in $C$, it follows that each $t_{ik}$ belongs to $\slMath_2(F)$. 
Moreover, $a= t_1t_2\cdots t_n$ implies that
\[
t_{01}t_{02}\cdots t_{0n}=0.
\]
For $i,j=1,2$, let $e_{ij}\in M_2(C)$ be the corresponding matrix unit. Writing each matrix $t_j$ in the basis $\{1,c_1,c_2,c_3,c_4\}$
and using that $c_ic_j=0$, the identity
$a= t_1t_2\cdots t_n$ can be seen to imply
\begin{align*}
e_{11}=&t_{11}t_{02}\cdots t_{0n} + t_{01}t_{12}t_{03}\cdots t_{0n} + \dots + t_{01}\cdots t_{0\,n-1} t_{1n}, \\
e_{12}=&t_{21}t_{02}\cdots t_{0n} + t_{01}t_{22}t_{03}\cdots t_{0n} + \dots + t_{01}\cdots t_{0\,n-1}t_{2n}, \\ 
e_{21}=  &t_{31}t_{02}\cdots t_{0n} + t_{01}t_{32}t_{03}\cdots t_{0n} + \dots + t_{01}\cdots t_{0\,n-1} t_{3n}, \\e_{22}= 
&t_{41}t_{02}\cdots t_{0n} + t_{01}t_{42}t_{03}\cdots t_{0n} + \dots + t_{01}\cdots t_{0\,n-1} t_{4n}.
\end{align*}
As the set $\{e_{11}, e_{12}, e_{21}, e_{22}\}$ is linearly independent in $M_2(F)$, this contradicts \autoref{le}. 
Therefore the matrices $t_1,\ldots,t_n$ do not exist, as desired.
\end{proof}

\section{Matrices over algebras with Bass stable rank one}
\label{sec:Bsr}

Given an algebra $A$ over an infinite field and $n\geq 3$, we show that every triangular $n$-by-$n$ matrix over $A$ is a product of two matrices with zero diagonal (\autoref{prp:TriangularFactorizes}), and hence a product of two commutators;
see \autoref{prp:TriangularProductCommutators}.
As an application, we show that every element in a von Neumann algebra of type~$\mathrm{I}_n$ is a product of two commutators;
see \autoref{prp:VNA}.

If $A$ has Bass stable rank one, we deduce that every invertible matrix over $A$ is a product of three commutators;
see \autoref{prp:ProdCommSR1}.
If $A$ is a right K-Hermite ring with Bass stable rank one, then every matrix over $A$ is a product of three commutators;
see \autoref{prp:ProdCommHermiteSR1}.

\begin{prp}
\label{prp:TriangularFactorizes}
Let $R$ be a unital ring, and let $n \geq 3$.
Then every upper triangular matrix in $M_n(R)$ is the product of two matrices with zero diagonals.
More precisely, if $a = (a_{j,k})_{j,k} \in M_n(R)$ is upper triangular, then $a=bc$ for the matrices $b=(b_{j,k})_{j,k} \in M_n(R)$ and $c=(c_{j,k})_{j,k} \in M_n(R)$ given by
\[
b_{j,k} = \begin{cases}
a_{j,n}, & \text{if } j\geq 2, k=1 \\
1, & \text{if } j=1, k=2 \\
a_{j,k-1}, & \text{if } k\geq 3 \\
0, & \text{else}
\end{cases},
\]
and all entries of $c$ zero except
\[
c_{2,1}=a_{1,1}, \quad
c_{2,n}=a_{1,n}, \quad
c_{1,n}=c_{3,2}=c_{4,3}=\ldots=c_{n,n-1}=1.
\]

Similarly, every lower triangular matrix in $M_n(R)$ is the product of two matrices with zero diagonals.
\end{prp}
\begin{proof}
The result for upper triangular matrices is proved by executing a matrix multiplication, and the result for lower triangular matrices is shown analogously.
We omit the details and instead indicate the factorizations for the cases $n = 3$ and $n = 4$.

In $M_3(R)$, we have:
\begin{align*}
\begin{pmatrix}
a_{11} & a_{12} & a_{13} \\
0 & a_{22} & a_{23} \\
0 & 0 & a_{33} \\
\end{pmatrix}
= \begin{pmatrix}
0 & 1 & a_{12} \\
a_{23} & 0 & a_{22} \\
a_{33} & 0 & 0
\end{pmatrix}
\begin{pmatrix}
0 & 0 & 1 \\
a_{11} & 0 & a_{13} \\
0 & 1 & 0
\end{pmatrix}.
\end{align*}

In $M_4(R)$, we have
\begin{align*}
\begin{pmatrix}
a_{11} & a_{12} & a_{13} & a_{14} \\
0 & a_{22} & a_{23} & a_{24} \\
0 & 0 & a_{33} & a_{34} \\
0 & 0 & 0 & a_{44}
\end{pmatrix}
= \begin{pmatrix}
0 & 1 & a_{12} & a_{13} \\
a_{24} & 0 & a_{22} & a_{23} \\
a_{34} & 0 & 0 & a_{33} \\
a_{44} & 0 & 0 & 0
\end{pmatrix}
\begin{pmatrix}
0 & 0 & 0 & 1 \\
a_{11} & 0 & 0 & a_{14} \\
0 & 1 & 0 & 0 \\
0 & 0 & 1 & 0
\end{pmatrix}.
\end{align*}

\end{proof}

The next result is well known, but we could not locate a precise reference.

\begin{lma}
\label{prp:ZeroDiagonal}
Let $n \geq 2$, and let $R$ be a unital ring containinig central elements $a_1,\ldots,a_n \in R$ such that the pairwise differences $a_j-a_k$ for $j \neq k$ are invertible in~$R$.
Then every $n$-by-$n$ matrix with zero diagonal is a commutator in $M_n(R)$.
\end{lma}
\begin{proof}
Consider the diagonal matrix $a$ with diagonal entries $a_1,\ldots,a_n$.
Given a matrix $b=(b_{jk})_{j,k} \in M_n(R)$, the commutator $[a,b]$ is the matrix $(c_{jk})_{j,k}$ with entries $c_{jk}=(a_j-a_k)b_{jk}$ for $j,k=1,\ldots,n$.
We illustrate the case $n=3$:
\begin{align*}
&\left[
\begin{pmatrix}
a_{1} & 0 & 0 \\
0 & a_{2} & 0 \\
0 & 0 & a_{3} \\
\end{pmatrix},
\begin{pmatrix}
b_{11} & b_{12} & b_{13} \\
b_{21} & b_{22} & b_{23} \\
b_{31} & b_{32} & b_{33} \\
\end{pmatrix}
\right] \\
&\qquad\qquad\qquad= \begin{pmatrix}
0 & (a_1-a_2)b_{12} & (a_1-a_3)b_{13} \\
(a_2-a_1)b_{21} & 0 & (a_2-a_3)b_{23} \\
(a_3-a_1)b_{31} & (a_3-a_2)b_{32} & 0
\end{pmatrix}.
\end{align*}

Now, given a matrix $c = (c_{jk})_{j,k} \in M_n(R)$ with zero diagonal, consider the matrix $b$ with entries $b_{jj} = 0$ for $j=1,\ldots,n$ and $b_{jk}:=(a_j-a_k)^{-1}c_{jk}$ for $j \neq k$.
Then $c=[a,b]$.
\end{proof}

\begin{thm}
\label{prp:TriangularProductCommutators}
Let $A$ be a unital algebra over an infinite field, and let $n \geq 3$.
Then every upper (lower) triangular matrix in $M_n(A)$ is the product of two commutators.
\end{thm}
\begin{proof}
By \autoref{prp:TriangularFactorizes}, every triangular matrix is the product of two matrices with zero diagonal.
Since $A$ is an algebra over an infinite field, the assumptions of \autoref{prp:ZeroDiagonal} are satisfied and it follows that every matrix over $A$ with zero diagonal is a commutator.
\end{proof}

For a topological space $X$, we write $C(X)$ for the algebra of all continuous functions $X\to\mathbb{C}$ endowed with pointwise operations. Recall that a space $X$ is said to be \emph{extremally disconnected} (also called a \emph{Stonean} space), if the closure of every open set in $X$ is open (and hence clopen).

\begin{exa}
\label{prp:VNA}
Let $n \geq 3$, and let $A$ be an $AW^*$-algebra of type $\mathrm{I}_n$ in the sense of \cite[Definition~18.2]{Ber72BearRgs}.
(This includes all von Neumann algebras of type $\mathrm{I}_n$, that is, von Neumann algebras such that every irreducible representation acts on a Hilbert space of dimension~$n$.)
We will argue that every element in $A$ is a product of two commutators. 

It is a standard fact in C*-algebra theory that there is an extremally disconnected compact Hausdorff space $X$ such that $A \cong M_n(C(X))$. Given $a \in M_n(C(X))$, by a result of Deckard and Pearcy \cite[Theorem~2]{DecPea63MatricesRgCtsFctsStonian} there exists a unitary $u \in M_n(C(X))$ such that $uau^*$ is upper triangular.
(A more conceptual proof of this result was given in \cite[Corollary~6]{Azo74BorelMeasurabilityLinAlg}.)
If $n \geq 3$, then it follows from \autoref{prp:TriangularProductCommutators} that $uau^*$ is a product of two commutators, and consequently so is $a$ itself.

The result also holds for $n=2$, and in fact for arbitrary von Neumann algebras of type $\mathrm{I}$, but the proof is more complicated since one needs to control the norm of the elements going into the commutators. This will appear in forthcoming work of the second and 
third named authors; see \cite{GarThi24pre:ProdCommutatorsVNA}.
\end{exa}

We say that a matrix $(a_{jk})_{j,k} \in M_n(R)$ has \emph{zero trace} if $a_{11} + \ldots + a_{nn} = 0$.
The following result is well known;
see, for example, \cite[Theorem~4]{KauPas14CommutatorsMatrices}.

\begin{thm}
\label{prp:TriangularTraceZero}
Let $R$ be a unital ring, and $n \geq 2$.
Then every triangular matrix in $M_n(R)$ with zero trace is a commutator.
\end{thm}

A unital ring $R$ is said to have \emph{Bass stable rank one} if for all $a,b \in R$ such that $R = Ra+Rb$, there exists $c \in R$ such that $R = R(a+cb)$. 
In other words, whenever $a$ and $b$ generate $R$ as a left ideal, then there exists an element $c\in R$ such that $a+cb$ is left invertible.
For more details and an overview on the theory of Bass stable rank, we refer to \cite{Vas84BassSR, Che11BookSR}.

Two matrices $a,b \in M_n(R)$ over a unital ring $R$ are said to be \emph{similar} if $a = vbv^{-1}$ for some $v \in \GL_n(R)$.
In \cite{VasWhe90CommutatorsSR1}, Vaserstein and Wheland showed that every invertible matrix over a ring with Bass stable rank one is a product of three triangular matrices, and similar to a product of two triangular matrices.
Combined with \autoref{prp:TriangularProductCommutators}, one can immediately deduce that invertible matrices over suitable rings are products of four commutators.
Using a more refined argument, we show that products of three commutators suffice;
see \autoref{prp:ProdCommSR1}.

\begin{lma}
\label{prp:FactorizationSR1}
Let $R$ be a unital ring of Bass stable rank one, let $n \geq 2$, and let $a \in \GL_n(R)$.
Then there exist $b,c \in \GL_n(R)$ such that $b$ is lower triangular, $c$ is upper triangular with all diagonal entries equal to $1$, and  $a$ is similar to $bc$.
\end{lma}
\begin{proof}
By \cite[Theorem~1]{VasWhe90CommutatorsSR1}, there exist $x,y,z \in \GL_n(R)$ such that $a = xyz$, and such that~$x$ and~$z$ are lower triangular, and $y$ is upper triangular.
From the proof of \cite[Theorem~1]{VasWhe90CommutatorsSR1} we see that we can arrange that $y$ and $z$ have all diagonal entries equal to $1$.
Set $b := zx$ and $c := y$.
Then $b$ is lower triangular, and $c$ is upper triangular with all diagonal entries equal to $1$.
Further, $a$ is similar to the matrix $zaz^{-1} = (zx)y = bc$.
\end{proof}

\begin{thm}
\label{prp:ProdCommSR1}
Let $A$ be a unital algebra over an infinite field, and assume that $A$ has Bass stable rank one.
Then for $n \geq 3$, every matrix in $\GL_n(A)$ is a product of three commutators.
\end{thm}
\begin{proof}
Let $a \in \GL_n(A)$.
Use \autoref{prp:FactorizationSR1} to find $b,c \in \GL_n(A)$ such that $b$ is lower triangular, $c$ is upper triangular with all diagonal entries equal to $1$, and~$a$ is similar to $bc$.
It suffices to show that $bc$ is a product of three commutators, since then so is~$a$.

Since $A$ is an algebra over an infinite field, we can find invertible elements $\lambda_1,\ldots,\lambda_n \in A$ such that $\lambda_1+\ldots+\lambda_n = 0$.
Let $e \in M_n(A)$ denote the diagonal matrix with diagonal entries $\lambda_1,\ldots,\lambda_n$.
Then $ec$ is upper triangular with diagonal $\lambda_1,\ldots,\lambda_n$. Thus $ec$ has trace zero, and is therefore a commutator by \autoref{prp:TriangularTraceZero}.
Further, $be^{-1}$ is lower triangular (not necessarily with trace zero), and therefore is a product of two commutators by \autoref{prp:TriangularProductCommutators}.
Thus, $bc=(be^{-1})(ec)$ is a product of three commutators.
\end{proof}

There are different notions of a `left (right) Hermite ring' in the literature, some meaning that every finitely generated, stably free left (right) $R$-module is free (see, for example, \cite[Definition~I.4.6]{Lam06SerresProblem}), and some referring to the notion studied by Kaplansky in \cite{Kap49ElemDivisors}.
Following Lam, \cite[Definition~I.4.23]{Lam06SerresProblem}, we say that a (not necessarily commutative) ring $R$ is \emph{right K-Hermite} (the 'K' standing for Kaplansky) if for every $1$-by-$2$ matrix $\begin{pmatrix} x & y \end{pmatrix} \in M_{1,2}(R)$ there exists $Q \in \GL_2(R)$ such that $\begin{pmatrix} x & y \end{pmatrix}Q = \begin{pmatrix} z & 0 \end{pmatrix}$ for some $z \in R$.
Equivalently, for every rectangular matrix $a \in M_{m,n}(R)$ there exists an invertible matrix $v \in M_n(R)$ such that $av$ is lower triangular;
see \cite[Theorem~3.5]{Kap49ElemDivisors}.
Similarly, a ring $R$ is \emph{left K-Hermite} if for every rectangular matrix $a \in M_{m,n}(R)$ there exists an invertible matrix $w \in M_m(R)$ such that $wa$ is upper triangular. 

The next result is analogous to \autoref{prp:FactorizationSR1}, with the only difference that we obtain a result for all matrices (not only invertible matrices), and the lower triangular matrix $b$ may thus not be invertible.

\begin{lma}
\label{prp:FactorizationHermiteSR1}
Let $R$ be a unital, right K-Hermite ring of Bass stable rank one, let $n \geq 2$, and let $a \in M_n(R)$.
Then there exist $b,c \in M_n(R)$ such that $b$ is lower triangular, $c$ is upper triangular with all diagonal entries equal to $1$, and  $a$ is similar to~$bc$.
\end{lma}
\begin{proof}
By \cite[Theorem~3.5]{Kap49ElemDivisors}, there exist a lower triangular matrix $x \in M_n(R)$ and $y \in \GL_n(R)$ such that $a = xy$.
We now apply \cite[Theorem~1]{VasWhe90CommutatorsSR1} for $y$ and obtain $u,v,w \in \GL_n(R)$ such that $y = uvw$, and such that~$u$ and~$w$ are lower triangular, and $v$ is upper triangular.
From the proof of \cite[Theorem~1]{VasWhe90CommutatorsSR1} we see that we can arrange that $v$ and $w$ have all diagonal entries equal to $1$.

Set $b := wxu$ and $c := v$.
Then $b$ is lower triangular, and $c$ is upper triangular with all diagonal entries equal to $1$.
Further, $a$ is similar to the matrix $waw^{-1} = w(xuvw)w^{-1} = (wxu)v = bc$.
\end{proof}

\begin{thm}
\label{prp:ProdCommHermiteSR1}
Let $A$ be a unital algebra over an infinite field, and assume that $A$ is right $K$-Hermite and has Bass stable rank one.
Then for $n \geq 3$, every matrix in $M_n(A)$ is a product of three commutators.
\end{thm}
\begin{proof}
This is analogous to the proof of \autoref{prp:ProdCommSR1}.
\end{proof}

\begin{qst}
Can the assumption that $A$ is an algebra over an infinite field be removed in \autoref{prp:ProdCommSR1} or \autoref{prp:ProdCommHermiteSR1}?
Do these results hold for $n=2$?
\end{qst}

\begin{rmk}
Chen and Chen showed in \cite[Theorem~2.2]{CheChe04ProdThreeTriangular} that a unital ring~$R$ is right K-Hermite and has Bass stable rank one if and only if every matrix $a \in M_n(R)$ admits a factorization $a=bcd$ in $M_n(R)$ with $b$ and $d$ lower triangular, $c$ upper triangular and all diagonal entries of $c$ and $d$ equal to $1$.
\end{rmk}

Let us point out a few instances to which the above results are applicable. 
The first one is extremely easy, but we will need it in the next section.

\begin{exa}
\label{lma:DivRingKHerSR1}
Every division ring $D$ is right K-Hermite and has Bass stable rank one. The latter is obvious since $D$ has no proper nonzero left ideals. To prove the former, take $x,y\in D$. We want to find an invertible matrix $Q\in M_2(D)$ such that $\begin{pmatrix} x & y \end{pmatrix}Q=\begin{pmatrix} z & 0 \end{pmatrix}$ 
for some $z\in D$.
If $x\neq 0$, one may take
$Q=\begin{psmallmatrix} 1 & -x^{-1}y \cr 0 & 1
\cr
\end{psmallmatrix}$. If $x=0$,
one may take $Q=\begin{psmallmatrix} 0 & 1 \cr 1 & 0
\cr
\end{psmallmatrix}$. It follows that $D$ is right K-Hermite. 
\end{exa}

The next example is more general.

\begin{exa}
A unital ring is said to be \emph{(von Neumann) regular} if for every $x \in R$ there exists $y \in R$ such that $x = xyx$.
If one can always arrange  $y$ to be invertible, then $R$ is said to be \emph{unit-regular}.
We refer to \cite{Goo79RegRings} for more details.

A regular ring has Bass stable rank one if and only if it is unit-regular;
see \cite[Proposition~4.12]{Goo79RegRings}.
Further, every unit-regular ring is right K-Hermite; this follows from \cite[Theorem~9]{MenMon82RegRgSR2} as noted in the introduction of \cite{AraGooOMePar97DiagonalizationRegRg}.

Thus, if $R$ is a unit-regular ring that is an algebra over an infinite field, and $n \geq 3$, then every matrix in $M_n(R)$ is a product of three commutators by \autoref{prp:ProdCommHermiteSR1}.
\end{exa}

\begin{exa}
A unital \ca{} $A$ is said to have \emph{stable rank one} if $\GL(A)$ is norm-dense in $A$;
see \cite{Rie83DimSRKThy}.
By \cite{HerVas84StableRangeCAlg}, $A$ has stable rank one if and only if $A$ has Bass stable rank one (as a ring).
Further, every \ca{} is an algebra over the infinite field of complex numbers.
Therefore, \autoref{prp:ProdCommSR1} applies to invertible matrices of size at least $3$-by-$3$ over \ca{s} of stable rank one.
In some cases, one has $A \cong M_n(B)$ for some $n \geq 3$ and some other \ca{} $B$ (which then automatically has stable rank one as well) and then \autoref{prp:ProdCommSR1} applies to invertible elements in $A$ itself.
For example, every invertible element in a UHF-algebra is a product of three commutators.

Many naturally occurring simple, unital \ca{s} have stable rank one. This includes all finite, nuclear, classifiable \ca{s} \cite{Ror04StableRealRankZ}; 
many finite, nuclear, non-classifiable \ca{s} \cite{EllHoTom09ClassSimpleSR1, Vil98SimpleCaPerforation, Tom08ClassificationNuclear}; 
reduced group \ca{s} of free products \cite{DykHaaRor97SRFreeProd}; 
and crossed products of minimal homeomorphisms on infinite, compact, metric spaces \cite{AlbLut22SRDiagASHCrProd}.

The comparison theory of positive elements and Hilbert modules is particularly well-developed for \ca{s} of stable rank one \cite{Thi20RksOps, AntPerRobThi22CuntzSR1}.
\end{exa}

\section{Matrices over division rings}
\label{sec:DivRgs}

In this section, we show that every matrix over a division ring with infinite center is a product of two commutators;
see \autoref{prp:ProdCommInftyCenter}.
We also show that every singular matrix over an arbitrary division ring is a product of two commutators;
see \autoref{prp:SingularMatrix}.

\begin{lma}
\label{prp:SimpleCommutatorSize2}
Let $D$ be a division ring, and let $r,s,t \in D$.
Then the matrix $a = 
\begin{psmallmatrix}
r & s \\
t & -r
\end{psmallmatrix}
\in M_2(D)$ is a commutator in $M_2(D)$.
More precisely, there exist $b \in \GL_2(D)$ and $c \in M_2(D)$ such that $a = [b,c]$.
\end{lma}
\begin{proof}
Case~1:
We have $r=0$ and $s,t \neq 0$.
Then
\begin{align*}
\begin{pmatrix}
0 & s \\
t & 0
\end{pmatrix}
= \left[
\begin{pmatrix}
0 & -s \\
t & 0
\end{pmatrix},
\begin{pmatrix}
1 & 0 \\
0 & 0
\end{pmatrix}
\right]
\end{align*}
and the matrix 
$\begin{psmallmatrix}
0 & -s \\
t & 0
\end{psmallmatrix}$
is invertible.

Case~2:
We have $s=0$.
Then
\begin{align*}
\begin{pmatrix}
r & 0 \\
t & -r
\end{pmatrix}
= \left[
\begin{pmatrix}
1 & 0 \\
1 & 1
\end{pmatrix},
\begin{pmatrix}
0 & -r \\
0 & -t
\end{pmatrix}
\right]
\end{align*}
and the matrix
$\begin{psmallmatrix}
1 & 0 \\
1 & 1
\end{psmallmatrix}$
is invertible.

Case~3:
We have $t=0$.
This is analogous to case~2.

Case~4:
We have $r,s,t \neq 0$.
Then
\begin{align*}
\begin{pmatrix}
r & s \\
t & -r
\end{pmatrix}
= \left[
\begin{pmatrix}
0 & -srt^{-1} \\
r & 0
\end{pmatrix},
\begin{pmatrix}
0 & -1 \\
0 & -tr^{-1}
\end{pmatrix}
\right]
\end{align*}
and the matrix
$\begin{psmallmatrix}
0 & -srt^{-1} \\
r & 0
\end{psmallmatrix}$
is invertible.
\end{proof}

Next, we consider arbitrary $2$-by-$2$ matrices over a division ring. 

\begin{prp}
\label{prp:ProdCommutatorsSize2}
Let $D$ be a division ring, and let $a \in M_2(D)$.
Then there exist $b,c,d,e \in M_2(D)$ such that $a=[b,c][d,e]$, and such that $[b,c]$ and $d$ are invertible.
In particular, every matrix in $M_2(D)$ is a product of two commutators.
\end{prp}
\begin{proof}
Let $a = 
\begin{psmallmatrix}
r & s \\
t & u
\end{psmallmatrix}
\in M_2(D)$.

\smallskip

Case~1: We have $s,t \neq 0$.
Then
\begin{align*}
a = \begin{pmatrix}
r & s \\
t & u
\end{pmatrix}
= 
\begin{pmatrix}
0 & -st^{-1} \\
1 & 0
\end{pmatrix}
\begin{pmatrix}
t & u \\
-ts^{-1}r & -t
\end{pmatrix}
\end{align*}
and the first matrix is invertible. By \autoref{prp:SimpleCommutatorSize2}, both matrices appearing in the factorization above are commutators of a matrix in $\GL_2(D)$ and a matrix in $M_2(D)$.

Case~2:
We have $s=0$ and $t \neq 0$.
Then
\begin{align*}
a = \begin{pmatrix}
r & 0 \\
t & u
\end{pmatrix}
= \begin{pmatrix}
1 & -(u-r)t^{-1} \\
0 & -1
\end{pmatrix}
\begin{pmatrix}
u & (u-r)t^{-1}u \\
-t & -u
\end{pmatrix}
\end{align*}
and the first matrix is invertible. Again by \autoref{prp:SimpleCommutatorSize2}, both matrices are commutators of a matrix in $\GL_2(D)$ and a matrix in $M_2(D)$.

Case~3:
We have $s \neq 0$ and $t=0$.
This is analogous to case~2.

Case~4:
We have $s=t=0$.
Then 
\begin{align*}
a = \begin{pmatrix}
r & 0 \\
0 & u
\end{pmatrix}
= \begin{pmatrix}
0 & 1 \\
1 & 0
\end{pmatrix}
\begin{pmatrix}
0 & u \\
r & 0
\end{pmatrix}
\end{align*}
and the first matrix is invertible. Once again by \autoref{prp:SimpleCommutatorSize2}, both matrices are commutators of a matrix in $\GL_2(D)$ and a matrix in $M_2(D)$.
\end{proof}

\begin{lma}
\label{prp:ProdCommutatorsCentral}
Let $D$ be a division ring containing at least three elements, let $n \geq 2$, and let $1_n \in M_n(D)$ denote the identity matrix.
Then there exist $b,c,d,e \in M_n(D)$ such that $1_n=[b,c][d,e]$, and such that $[b,c]$ and $d$ are invertible.
\end{lma}
\begin{proof}
For $n=2$ this follows from \autoref{prp:ProdCommutatorsSize2}, so
we consider the case $n=3$.

Since $D$ contains at least three elements, we can choose $x\in D\setminus\{0,1\}$. Then $y:=x-1$ is not zero.
We have
\begin{align*}
\left[
\begin{pmatrix}
0 & 0 & -x \\
1 & 0 & 0 \\
0 & y & 0 \\
\end{pmatrix},
\begin{pmatrix}
0 & 0 & 1 \\
0 & 0 & 0 \\
0 & -1 & 0
\end{pmatrix}
\right]	
= 
\begin{pmatrix}
0 & 1 & 0 \\
0 & 0  & 1 \\
1 & 0 & 0
\end{pmatrix}
\end{align*}
and the matrix 
$\begin{psmallmatrix}
0 & 0 & -x \\
1 & 0 & 0 \\
0 & y & 0 \\
\end{psmallmatrix}$
is invertible.
Similarly, we see that 
$\begin{psmallmatrix}
0 & 0 & 1 \\
1 & 0 & 0 \\
0 & 1 & 0 \\
\end{psmallmatrix}$
is an (invertible) commutator in $M_3(D)$. Since
$1_3=\begin{psmallmatrix}
0 & 1 & 0 \\
0 & 0  & 1 \\
1 & 0 & 0
\end{psmallmatrix}\begin{psmallmatrix}
0 & 0 & 1 \\
1 & 0 & 0 \\
0 & 1 & 0 \\
\end{psmallmatrix}$, this establishes the case $n=3$.

In preparation for the general case, let us fix matrices $b_2,c_2,d_2,e_2\in M_2(D)$ and $b_3,c_3,d_3,e_3\in M_3(D)$ satisfying
\[1_2=\big[b_2,c_2\big]\big[d_2,e_2\big] \ \ \mbox{ and } \ \ 1_3=\big[b_3,c_3\big]\big[d_3,e_3\big],\]
and such that
$\big[b_2,c_2\big], \big[b_3,c_3\big], d_2$ and $d_3$ are invertible.
Given $n\geq 4$, find $k,l\ge 0$ with $n = 2k + 3l$.
Let $b\in M_n(D)$ be the block-diagonal matrix with $k$ blocks $b_2$ and $l$ blocks $b_3$. Define $c,d,e\in M_n(D)$ similarly. It is then easy to check that $1_n=[b,c][d,e]$, and that $[b,c]$ and $d$ are invertible, thus finishing the proof.
\end{proof}

\begin{thm}
\label{prp:ProdCommInftyCenter}
Let $D$ be a division ring with infinite center.
Then every matrix in $M_n(D)$ for $n \geq 2$ is a product of two commutators.
\end{thm}
\begin{proof}
For every (not necessarily infinite) 
 field $F$, every matrix in $M_n(F)$ for $n \geq 2$ is a product of two commutators; see \cite[Theorem 4.1]{Bot97ProdMatPrescribedTraces}.
Thus, we may assume that $D$ is noncommutative.

We verify the following stronger result by induction over $n$:
For all $a \in M_n(D)$, there exist $b,c,d,e \in M_n(D)$ such that $a=[b,c][d,e]$, and such that $[b,c]$ and $d$ are invertible.

The case $n=2$ follows from \autoref{prp:ProdCommutatorsSize2}.
Assume that the result holds for some $n \geq 2$, and let us verify it for $n+1$.

Let $a \in M_{n+1}(D)$.
If $a$ is central, then the result follows from \autoref{prp:ProdCommutatorsCentral}.
Thus, we may assume that $a$ is noncentral.
Then, by \cite[Proposition~1.8]{AmiRow94RedTr0}, $a$ is similar to a matrix whose $(1,1)$-entry is zero.
(Note that the global assumption of \cite{AmiRow94RedTr0} that division rings are finite-dimensional over their centers is not used in the proof of \cite[Proposition~1.8]{AmiRow94RedTr0}.)
Since the desired conclusion is invariant under similarity, we may assume, without loss of generality, that $a_{11}=0$.
Let $b \in M_{1,n}(D)$, $c \in M_{n,1}(D)$ and $x \in M_n(D)$ satisfy
\[
a = \begin{pmatrix}
0 & b \\
c & x
\end{pmatrix}.
\]
Since $D$ is noncommutative, there exist a nonzero commutator $d \in D$.
By the inductive assumption, we have $x=y[v,w]$ for an invertible commutator $y \in M_n(D)$ and $v \in \GL_n(D)$ and $w \in M_n(D)$.
Then
\[
a 
= \begin{pmatrix}
0 & b \\
c & x
\end{pmatrix}
= \begin{pmatrix}
d & 0 \\
0 & y
\end{pmatrix}
\begin{pmatrix}
0 & d^{-1}b \\
y^{-1}c & [v,w]
\end{pmatrix}.
\]
The matrix 
$\begin{psmallmatrix}
d & 0 \\
0 & y
\end{psmallmatrix}$
is an invertible commutator in $M_{n+1}(D)$.
It remains to verify that 
$\begin{psmallmatrix}
0 & d^{-1}b \\
y^{-1}c & [v,w]
\end{psmallmatrix}$
is the commutator of some matrix in $\GL_{n+1}(D)$ and a matrix in $M_{n+1}(D)$. For this, we will need a result of \cite{Coh73SimReductionOverSkewField}, and we first recall some of its terminology. 

An element $\lambda \in D$ is called a \emph{left eigenvalue} of $v$ if there exists a nonzero $\xi \in M_{n,1}(D)$ such that $v\xi = \xi\lambda$, and $\lambda$ is called a \emph{right eigenvalue} if there exists a nonzero $\eta \in M_{1,n}(D)$ such that $\eta v = \lambda\eta$.
The set of all left and right eigenvalues is called the spectrum of $v$;
see \cite{Coh73SimReductionOverSkewField}.
By \cite[Proposition~2.5]{Coh73SimReductionOverSkewField}, an element~$\lambda$ in the center $Z(D)$ of $D$ belongs to the spectrum of $v$ if and only if $v-\lambda$ is singular ($\lambda$ is called a `singular eigenvalue' of $z$).
Further, by \cite[Theorem~2.4]{Coh73SimReductionOverSkewField}, the spectrum of $v$ contains at most finitely many conjugacy classes.
Consequently, there are at most finitely many $\lambda \in Z(D)$ such that $v-\lambda$ is singular.

Using that $Z(D)$ is infinite, we obtain a nonzero $\lambda \in Z(D)$ such that $v-\lambda$ is invertible.
We then have
\[
\begin{pmatrix}
0 & d^{-1}b \\
y^{-1}c & [v,w]
\end{pmatrix}
= \left[
\begin{pmatrix}
\lambda & 0 \\
0 & v
\end{pmatrix},
\begin{pmatrix}
0 & d^{-1}b(\lambda-v)^{-1} \\
(v-\lambda)^{-1}y^{-1}c & w
\end{pmatrix}
\right]
\]
and 
$\begin{psmallmatrix}
\lambda & 0 \\
0 & v
\end{psmallmatrix}$
is invertible.
This proves the inductive step and finishes the proof.
\end{proof}

Recall that a division ring that is finite-dimensional over its center is called a {\em central division algebra}.
Since every finite division ring is a field, \autoref{prp:ProdCommInftyCenter} along with \cite[Theorem~4.1]{Bot97ProdMatPrescribedTraces} yields the following result.

\begin{cor}
\label{cd}
Let $D$ be a central division algebra.
Then every matrix in $M_n(D)$ for $n \geq 2$ is a product of two commutators.    
\end{cor}

The comparison of \autoref{prp:ProdCommutatorsSize2} and \autoref{prp:ProdCommInftyCenter} raises the following question:

\begin{qst}
Can the assumption that $D$ has infinite center be removed in \autoref{prp:ProdCommInftyCenter}?
\end{qst}

For $n=2$, the answer is ``yes'' by \autoref{prp:ProdCommutatorsSize2}. The next proposition provides another such instance.

\begin{prp}
\label{prp:SingularMatrix}
Every singular matrix over a division ring is a product of two commutators.
\end{prp}
\begin{proof}
Let $D$ be a division ring, let $n \geq 2$, and let $a \in M_n(D)$ be non-invertible.
By \autoref{prp:ProdCommutatorsSize2}, we may assume that $n \geq 3$.
We may also assume that $D$ contains at least three elements, since otherwise $D$ is a field and then every matrix over $D$ is a product of two commutators by \cite[Theorem 4.1]{Bot97ProdMatPrescribedTraces}.

Since $D$ is a right K-Hermite ring and has Bass stable rank one by \autoref{lma:DivRingKHerSR1}, we can apply \autoref{prp:FactorizationHermiteSR1} and deduce that $a$ is similar to the product $bc$ for a lower triangular matrix $b$ and an upper triangular matrix $c$ with all diagonal entries equal to $1$.
Since the statement is invariant under similarity, we may assume that $a=bc$.
Further, since $a=bc$ is not invertible, using that $D$ is a division ring it follows that at least one of the diagonal entries of $b$ is zero. 
Without loss of generality, upon taking a similar matrix we may assume that $b_{nn}=0$.

Using that $D$ contains at least three elements and $n \geq 3$, we can choose nonzero $e_1,\ldots,e_{n-1} \in D$ such that $e_1+\ldots+e_{n-1}=0$.
Let $e \in M_n(D)$ be the diagonal matrix with diagonal entries $e_1,\ldots,e_{n-1},1$.
Then $a=bc=(be^{-1})(ec)$, and the matrix $be^{-1}$ is lower diagonal with diagonal entries $b_{1,1}e_1^{-1},\ldots,b_{n-1,n-1}e_{n-1}^{-1}, b_{nn}$. Similarly, $ec$ is upper diagonal with diagonal entries $e_1,\ldots,e_{n-1},1$.

Let $b'$ be equal to the matrix $be^{-1}$, except with the $(n,n)$-entry replaced by $-\sum_{j=1}^{n-1}b_{j,j}e_j^{-1}$; and let $c'$ be equal to the matrix $ec$, except with the $(n,n)$-entry replaced by $0$.
Then $a=b'c'$, and $b'$ and $c'$ are triangular matrices with zero trace, therefore commutators by \autoref{prp:TriangularTraceZero}.
The factorization is:
\begin{align*}
a
&= \begin{pmatrix}
b_{11} & 0 & 0 & \cdots & 0 \\
b_{21} & b_{22} & 0 & & \vdots \\
\vdots &  & \ddots \\
b_{n-1,1} & \ldots & & b_{n-1,n-1} & 0 \\
b_{n,1} & \ldots & & b_{n,n-1} & 0
\end{pmatrix}
\begin{pmatrix}
1 & c_{1,2} & c_{1,3} & \cdots & c_{1,n} \\
0 & 1 & c_{2,3} & \cdots & c_{2,n} \\
\vdots & & \ddots & & \vdots \\
0 & \ldots & & 1 & c_{n-1,n} \\
0 & \ldots & & 0 & 1
\end{pmatrix} \\
&= \begin{pmatrix}
b_{11}e_1^{-1} & 0 & 0 & \cdots & 0 \\
\ast & b_{22}e_2^{-1} & 0 & & \vdots \\
\vdots &  & \ddots \\
\ast & \ldots & & b_{n-1,n-1}e_{n-1}^{-1} & 0 \\
\ast & \ldots & & \ast & 0
\end{pmatrix}
\begin{pmatrix}
e_1 & \ast & \ast & \cdots & \ast \\
0 & e_2 & \ast & \cdots & \ast \\
\vdots & & \ddots & & \vdots \\
0 & \ldots & & e_{n-1} & \ast \\
0 & \ldots & & 0 & 1
\end{pmatrix} \\
&= \begin{pmatrix}
b_{11}e_1^{-1} & 0 & \cdots & 0 \\
\vdots & \ddots & & \vdots \\
\ast & \ldots & b_{n-1,n-1}e_{n-1}^{-1} & 0 \\
\ast & \ldots & \ast & -\sum_{j=1}^{n-1} b_{jj}e_j^{-1}
\end{pmatrix}
\begin{pmatrix}
e_1 & \ast & \cdots & \ast \\
\vdots & \ddots & & \vdots \\
0 & \ldots & e_{n-1} & \ast \\
0 & \ldots & 0 & 0
\end{pmatrix}.
\end{align*}
\end{proof}

\section{Commutators with a fixed element}
\label{sec:FixedElement}

In this section, we consider the more general problem of presenting elements in matrix algebras by commutators with a fixed matrix $a$. This is obviously considerably more demanding than allowing arbitrary commutators, so we will restrict ourselves to matrices over a field $F$.
Our goal is to prove Theorem~D from the introduction.

We remark that if a matrix $a\in M_n(F)$ has rank $k$, then any commutator $[a,x]$, with $x\in M_n(F)$, has rank at most $2k$. The same is therefore true for any product 
$[a,x_1]\cdots [a,x_m]$, with $x_i\in M_n(F)$. 
In order to represent every matrix in $M_n(F)$ by commutators 
$[a,x]$, their products   are thus insufficient and we are forced to involve  sums of products. 
Motivated by the invariant $\xi$ from \cite[Definition~5.1]{GarThi23arX:GenByCommutators} (see the introduction), we are particularly interested in  sums of products of two commutators. Another motivation is the result by Mesyan \cite[Theorem~15]{Mes06CommutatorRings} which states that every
trace zero matrix can be written as a sum of two commutators with
fixed matrices.

Our approach is based  on the concept of a derivation. 
Recall that a linear map~$D$ from an algebra $A$ to itself is called a {\em derivation} if $D(xy)=D(x)y+ xD(y)$ for all $x,y\in A$. 
For any $a\in A$, the map $x\mapsto [a,x]$ is a derivation. Such derivations are called {\em inner}. The problem that we address can obviously be formulated in terms of inner derivations.

Let us start with an observation which is implicit in Herstein's paper \cite{Her78NoteDerivations}.
Let~$A$ be any algebra and let $D\colon A \to A$ be a derivation. 
A straightforward verification shows that for all $x,y,z\in A$, we have
\[
xD^3(y)z = D\left(xD^2(y)z\right)  - D(x)D\left(D(y)z\right) - D\left(xD(y)\right)D(z)+ 2D(x)D(y)D(z).
\]
Accordingly, if $D^3\ne 0$, then the subalgebra $\overline{D(A)}$ generated 
by the image of $D$ contains a nonzero ideal of $A$. In particular, $\overline{D(A)}$ is equal to the whole algebra $A$ if $A$ is simple. More precisely,  the above formula shows that  every element in $A$ is a sum of products of at most three elements from the image of $D$. 

It should be remarked that the assumption that $D^3\ne 0$ is necessary.
Indeed, every element $a\in A$ such that $a^2=0$ gives rise to the inner derivation $D(x)=[a,x]$ which, as can be easily checked, satisfies  $D^3=0$ and $a\overline{D(A)} a=\{0\}$. The latter implies that $\overline{D(A)}$ cannot be equal to $A$ if
$a\ne 0$ and $A$ is simple.

The above observation, however, does not help us if we wish to present every element in $A$ as a sum of products of exactly two elements from the image of an (inner) derivation $D$. 
A slightly more sophisticated approach is necessary to tackle this problem. 
We start with the following result.

\begin{lma}
\label{lder} 
Let $D$ be a derivation of an algebra $A$, and let $b,c\in A$ satisfy $D(b)c=0$. Then
\[
xD(b)D(c)z = D(xb)D(cz) - D(x)D(bcz)
\]
for all $x,z\in A$.
Therefore, if $D(b)D(c)\ne 0$ and $A$ is simple, then every element in $A$ is a sum of  elements of the form $D(x)D(y)$ with $x,y\in A$.
\end{lma}
\begin{proof}
Note that $D(b)c =0$ implies 
\[
D(xb)D(cz)= D(x)bcD(z) + xD(b)D(c)z + D(x)bD(c)z
\]
and
\[
D(bcz)= bD(c)z + bcD(z),
\]
from which the formula from the statement of the lemma follows. 
If $D(b)D(c)\ne 0$, then this formula implies that the ideal of $A$ generated by $D(b)D(c)$ is contained in the set of sums of elements of the form $D(x)D(y)$ with $x,y\in A$.
Therefore, this set is equal to $A$ if $A$ is simple.
\end{proof}

\autoref{lder} raises the question of when do there exist elements $b,c\in A$ such that $D(b)c =0$ and $D(b)D(c)\ne 0$. 
In light of our goal, we are interested in the case where $D$ is an inner derivation and $A=M_n(F)$.
We will consider a somewhat more general situation in \autoref{lder2}. 
To this end, we need a result of general interest which is almost certainly known. 
However, we were unable to find a reference that would cover vector spaces over arbitrary fields. 
We therefore provide a proof which was shown to us by Cl{\'e}ment de Seguins Pazzis, who kindly allowed us to include it here.

Recall that an endomorphism $a$ of an $F$-vector space $V$ is said to be \emph{algebraic} if there exists a nonzero polynomial $p\in F[X]$ with coefficients in $F$ such that $p(a)=0$. Moreover, the \emph{degree} of $a$ is the smallest degree of such a polynomial.

\begin{lma}
\label{lc} 
Let $n$ be a natural number. An endomorphism  $a$ of a vector space  $V$ (over any  field) is algebraic of degree at most $n$ if and only if  the set $\{v, av, \dots,a^n v\}$ is linearly dependent for every $v\in V$.
\end{lma}
\begin{proof}
It suffices to prove the  `if' part. 
Thus, assume that the set $\{v, av, \dots,a^n v\}$ is linearly dependent for each $v\in V$. 
Denote by $V_v$ the linear span of this set, and by $p_v$ the minimal polynomial of the restriction of $a$ to $V_v$. 
Pick $v_0\in V$ such that $p_{v_0}$ has maximal degree. 
Our goal is to show that $p_{v_0}(a) =0$. 
Since the dimension of $V_{v_0}$ is at most $n$ by our assumption, this will prove the result.

Fix $v\in V$ and let us show that $p_{v_0}(a)v=0$. 
Let $\tilde{a}$ denote the restriction of $a$ to $V_{v_0}+ V_v$, and let
$\tilde{p}$ 
be the minimal polynomial of $\tilde{a}$. Since  $V_{v_0}\subseteq V_{v_0}+ V_v$, $p_{v_0}$ divides $\tilde{p}$. If 
$\tilde{p}$ was equal to 
$p_w$ for some $w\in V$, then it would follow, in view of the choice of $v_0$,
that $p_{v_0}=\tilde{p}$ and hence $p_{v_0}(a)v=0$, as desired.

The fact that $\tilde{p}$ is really equal to 
$p_w$ for some $w\in V$ follows by examining the Frobenius canonical form of  $\tilde{a}$. Indeed, $\tilde{a}$
can be represented in some basis as a block-diagonal matrix with blocks being companion matrices whose associated polynomials form a sequence that is non-increasing with respect to the divisibility relation. The first polynomial  in the sequence is the minimal polynomial $\tilde{p}$, and, denoting the degree of  $\tilde{p}$ by $d$,
the first $d$ vectors in 
the basis are $w,aw,\dots, a^{d-1}w$ for some $w\in \tilde{V}$. Since these vectors are linearly independent, the degree of
$p_w$ is at least $d$. On the other hand, $p_w$  divides $\tilde{p}$ since $w\in \tilde{V}$. 
Therefore, $\tilde{p}=p_w$.
\end{proof}

\begin{lma}
\label{lder2} 
Let $A $ be the algebra of all endomorphisms of the vector space $V$. 
Let $a\in A$ and let $D$ be the inner derivation given by $D(x) = [a,x]$. 
If $a$ is not algebraic of degree at most $2$, then there exists an element $b \in A$ such that $D(b)b =0$ and $D(b)^2\ne 0$.
\end{lma}
\begin{proof} 
In light of our assumption, \autoref{lc} shows that there exists $v\in V$ such that $a^2v$ does not lie in the linear span of $\{v, av\}$. 
Therefore, there is a linear functional $f$ on $V$ such that $f(v)= f(av)=0$ and $f(a^2v)=1$. 
Let $b$ be the rank one endomorphism defined by $bu= f(u)v$ for all $u\in V$. 
Observe that $b^2=bab=0$ and $ba^2b=b$. 
Consequently, $D(b)b= ab^2 - bab=0$ and $D(b)^2 = (ab - ba)^2 = -ba^2 b = -b\ne 0$.
\end{proof}

We are now ready to prove the main result of the section.
We note that the number of summands needed in statement~(2) 
is at most $n^2$, since this is the linear dimension of $M_n(F)$.
It is conceivable that the smallest number of summands needed may be related to the degree
of the minimal polynomial of the matrix $a$, but we have not explored 
this any further.
%

\begin{thm}
\label{tder} 
Let $F$ be a field, let $n\ge 2$, and let $a\in M_n(F)$. 
The following two conditions are equivalent:
\begin{enumerate} 
\item
The degree of the minimal polynomial of $a$ is greater than $2$.
\item
Every element in $M_n(F)$ can be written as a sum of elements of the form $[a,x][a,y]$ with $x,y\in M_n(F)$. 
\end{enumerate}
\end{thm}
\begin{proof} 
Let us show that~(1) implies~(2).
Condition~(1) can be read as saying that~$a$ is not algebraic of degree at most~$2$. If we denote by $D$ the inner derivation given by $D(c)=[a,c]$ for $c \in M_n(F)$, then by \autoref{lder2} there exists an element $b\in A$ such that $D(b)b =0$ and $D(b)^2\ne 0$.

Since $D(b)^2 \neq 0$, and since the algebra $M_n(F)$ is simple, we find a natural number $M$ and elements $r_j,s_j \in M_n(F)$ for $j=1,\ldots,M$ such that
\[
1 
= \sum_{j=1}^M r_jD(b)^2s_j.
\]
Given $x \in M_n(F)$, it follows from \autoref{lder} that 
\begin{align*}
x 
&= \sum_{j=1}^M xr_jD(b)^2s_j
= \sum_{j=1}^M \big( D(xr_jb)D(bs_j) - D(x)D(b^2s_j) \big) \\
&= \sum_{j=1}^M \big( [a,xr_jb][a,bs_j] + [a,x][a,-b^2s_j] \big).
\end{align*}
This proves (2).

In order to show the converse, assume that~(1) does not hold and let us show that~(2) does not hold either. 
The case where the degree of the minimal polynomial of~$a$ is~$1$ is trivial, so we may assume that it is equal to $2$. 
Let $\bar{F}$ denote the algebraic closure of $F$ and let $\lambda,\mu \in \bar{F}$ satisfy $(a-\lambda 1_n)(a-\mu 1_n)=0$. 
Using at the first step that $\lambda 1_n$ and $\mu 1_n$ commute with all the elements of $M_n(F)$, for all $x,y\in M_n(F)$ we get
\begin{align*} 
&[a,x][a,y] 
= [(a-\lambda 1_n),x][(a-\mu 1_n),y] \\
&\qquad = (a-\lambda 1_n)x(a-\mu 1_n)y - (a-\lambda 1_n)xy(a-\mu 1_n) + x(a-\lambda 1_n) y (a-\mu 1_n).
\end{align*}
This implies that
\[
(a-\mu 1_n)[a,x][a,y] (a-\lambda 1_n)=0. 
\]
Denoting by $S$ the set of sums of elements of the form $[a,x][a,y]$, we thus have
\[
(a-\mu 1_n)S (a-\lambda 1_n)=\{0\}.
\]
Assuming that $S= M_n(F)$, it follows that $(a-\mu 1_n)e_{ij} (a-\lambda 1_n)=0 $ for every matrix unit $e_{ij}$ in $M_n(F)$, which in turn implies that
\[
(a-\mu 1_n)z_{ij}e_{ij} (a-\lambda 1_n)=0 
\]
for every $z_{ij}\in \bar{F}$. 
Thus, we deduce that $(a-\mu 1_n)M_n(\bar{F}) (a-\lambda 1_n)=\{0\}$. 
However, this is impossible since $a-\mu 1_n$ and $a-\lambda 1_n$ are nonzero matrices of $M_n(\bar{F})$.
Therefore, $S\ne M_n(F)$.
\end{proof}


\begin{thebibliography}{KBMRY20}

\bibitem[AM57]{AlbMuc57MatricesTraceZero}
\bgroup\scshape{}A.~A. Albert\egroup{} and
  \bgroup\scshape{}B.~Muckenhoupt\egroup{}, On matrices of trace zeros,
  \emph{Michigan Math. J.} \textbf{4} (1957), 1--3.

\bibitem[AL22]{AlbLut22SRDiagASHCrProd}
\bgroup\scshape{}M.~Alboiu\egroup{} and \bgroup\scshape{}J.~Lutley\egroup{},
  The stable rank of diagonal {ASH} algebras and crossed products by minimal
  homeomorphisms,  \emph{M\"{u}nster J. Math.} \textbf{15} (2022), 167--220.

\bibitem[AR94]{AmiRow94RedTr0}
\bgroup\scshape{}S.~A. Amitsur\egroup{} and \bgroup\scshape{}L.~H.
  Rowen\egroup{}, Elements of reduced trace {$0$},  \emph{Israel J. Math.}
  \textbf{87} (1994), 161--179.

\bibitem[APRT22]{AntPerRobThi22CuntzSR1}
\bgroup\scshape{}R.~Antoine\egroup{}, \bgroup\scshape{}F.~Perera\egroup{},
  \bgroup\scshape{}L.~Robert\egroup{}, and \bgroup\scshape{}H.~Thiel\egroup{},
  \ca{s} of stable rank one and their {C}untz semigroups,  \emph{Duke Math. J.}
  \textbf{171} (2022), 33--99.

\bibitem[AGOP97]{AraGooOMePar97DiagonalizationRegRg}
\bgroup\scshape{}P.~Ara\egroup{}, \bgroup\scshape{}K.~R. Goodearl\egroup{},
  \bgroup\scshape{}K.~C. O'Meara\egroup{}, and
  \bgroup\scshape{}E.~Pardo\egroup{}, Diagonalization of matrices over regular
  rings,  \emph{Linear Algebra Appl.} \textbf{265} (1997), 147--163.

\bibitem[Azo74]{Azo74BorelMeasurabilityLinAlg}
\bgroup\scshape{}E.~A. Azoff\egroup{}, Borel measurability in linear algebra,
  \emph{Proc. Amer. Math. Soc.} \textbf{42} (1974), 346--350.

\bibitem[Ber72]{Ber72BearRgs}
\bgroup\scshape{}S.~K. Berberian\egroup{}, \emph{Baer *-rings},
  Springer-Verlag, New York-Berlin, 1972, Die Grundlehren der mathematischen
  Wissenschaften, Band 195.

\bibitem[Bot97]{Bot97ProdMatPrescribedTraces}
\bgroup\scshape{}J.~D. Botha\egroup{}, Products of matrices with prescribed
  nullities and traces,  \emph{Linear Algebra Appl.} \textbf{252} (1997),
  173--198.

\bibitem[Che11]{Che11BookSR}
\bgroup\scshape{}H.~Chen\egroup{}, \emph{Rings related to stable range
  conditions}, \emph{Series in Algebra} \textbf{11}, World Scientific
  Publishing Co. Pte. Ltd., Hackensack, NJ, 2011.

\bibitem[CC04]{CheChe04ProdThreeTriangular}
\bgroup\scshape{}H.~Chen\egroup{} and \bgroup\scshape{}M.~Chen\egroup{}, On
  products of three triangular matrices over associative rings,  \emph{Linear
  Algebra Appl.} \textbf{387} (2004), 297--311.

\bibitem[Coh73]{Coh73SimReductionOverSkewField}
\bgroup\scshape{}P.~M. Cohn\egroup{}, The similarity reduction of matrices over
  a skew field,  \emph{Math. Z.} \textbf{132} (1973), 151--163.

\bibitem[DP63]{DecPea63MatricesRgCtsFctsStonian}
\bgroup\scshape{}D.~Deckard\egroup{} and \bgroup\scshape{}C.~Pearcy\egroup{},
  On matrices over the ring of continuous complex valued functions on a
  {S}tonian space,  \emph{Proc. Amer. Math. Soc.} \textbf{14} (1963), 322--328.

\bibitem[DHR97]{DykHaaRor97SRFreeProd}
\bgroup\scshape{}K.~Dykema\egroup{}, \bgroup\scshape{}U.~Haagerup\egroup{}, and
  \bgroup\scshape{}M.~R{\o}rdam\egroup{}, The stable rank of some free product
  \ca{s},  \emph{Duke Math. J.} \textbf{90} (1997), 95--121.

\bibitem[EHT09]{EllHoTom09ClassSimpleSR1}
\bgroup\scshape{}G.~A. Elliott\egroup{}, \bgroup\scshape{}T.~M. Ho\egroup{},
  and \bgroup\scshape{}A.~S. Toms\egroup{}, A class of simple \ca{s} with
  stable rank one,  \emph{J. Funct. Anal.} \textbf{256} (2009), 307--322.

\bibitem[GT23]{GarThi23arX:GenByCommutators}
\bgroup\scshape{}E.~Gardella\egroup{} and \bgroup\scshape{}H.~Thiel\egroup{},
  Rings and \ca{s} generated by commutators, preprint (arXiv:2301.05958
  [math.RA]), 2023.

\bibitem[GT24]{GarThi24pre:ProdCommutatorsVNA}
\bgroup\scshape{}E.~Gardella\egroup{} and \bgroup\scshape{}H.~Thiel\egroup{},
  Products of additive commutators in von {N}eumann algebras, in preparation,
  2024.

\bibitem[Goo79]{Goo79RegRings}
\bgroup\scshape{}K.~R. Goodearl\egroup{}, \emph{von {N}eumann regular rings},
  \emph{Monographs and Studies in Mathematics} \textbf{4}, Pitman (Advanced
  Publishing Program), Boston, Mass.-London, 1979.

\bibitem[HV84]{HerVas84StableRangeCAlg}
\bgroup\scshape{}R.~H. Herman\egroup{} and \bgroup\scshape{}L.~N.
  Vaserstein\egroup{}, The stable range of \ca{s},  \emph{Invent. Math.}
  \textbf{77} (1984), 553--555.

\bibitem[Her78]{Her78NoteDerivations}
\bgroup\scshape{}I.~N. Herstein\egroup{}, A note on derivations,  \emph{Canad.
  Math. Bull.} \textbf{21} (1978), 369--370.

\bibitem[KBMRY20]{KanMalRowYav20EvalNcPolynomials}
\bgroup\scshape{}A.~Kanel-Belov\egroup{}, \bgroup\scshape{}S.~Malev\egroup{},
  \bgroup\scshape{}L.~Rowen\egroup{}, and \bgroup\scshape{}R.~Yavich\egroup{},
  Evaluations of noncommutative polynomials on algebras: methods and problems,
  and the {L}'vov-{K}aplansky conjecture,  \emph{SIGMA Symmetry Integrability
  Geom. Methods Appl.} \textbf{16} (2020), Paper No. 071, 61.

\bibitem[Kap49]{Kap49ElemDivisors}
\bgroup\scshape{}I.~Kaplansky\egroup{}, Elementary divisors and modules,
  \emph{Trans. Amer. Math. Soc.} \textbf{66} (1949), 464--491.

\bibitem[KP14]{KauPas14CommutatorsMatrices}
\bgroup\scshape{}M.~Kaufman\egroup{} and \bgroup\scshape{}L.~Pasley\egroup{},
  On commutators of matrices over unital rings,  \emph{Involve} \textbf{7}
  (2014), 769--772.

\bibitem[Lam06]{Lam06SerresProblem}
\bgroup\scshape{}T.~Y. Lam\egroup{}, \emph{Serre's problem on projective
  modules}, \emph{Springer Monographs in Mathematics}, Springer-Verlag, Berlin,
  2006.

\bibitem[MM82]{MenMon82RegRgSR2}
\bgroup\scshape{}P.~Menal\egroup{} and \bgroup\scshape{}J.~Moncasi\egroup{}, On
  regular rings with stable range {$2$},  \emph{J. Pure Appl. Algebra}
  \textbf{24} (1982), 25--40.

\bibitem[Mes06]{Mes06CommutatorRings}
\bgroup\scshape{}Z.~Mesyan\egroup{}, Commutator rings,  \emph{Bull. Austral.
  Math. Soc.} \textbf{74} (2006), 279--288.

\bibitem[Rie83]{Rie83DimSRKThy}
\bgroup\scshape{}M.~A. Rieffel\egroup{}, Dimension and stable rank in the
  {$K$}-theory of \ca{s},  \emph{Proc. London Math. Soc. (3)} \textbf{46}
  (1983), 301--333.

\bibitem[R{\o}r04]{Ror04StableRealRankZ}
\bgroup\scshape{}M.~R{\o}rdam\egroup{}, The stable and the real rank of
  {$\mathcal{Z}$}-absorbing \ca{s},  \emph{Internat. J. Math.} \textbf{15}
  (2004), 1065--1084.

\bibitem[Sho37]{Sho37SaetzeMatrizen}
\bgroup\scshape{}K.~Shoda\egroup{}, Einige {S}\"{a}tze \"{u}ber {M}atrizen,
  \emph{Jpn. J. Math.} \textbf{13} (1937), 361--365.

\bibitem[Thi20]{Thi20RksOps}
\bgroup\scshape{}H.~Thiel\egroup{}, Ranks of operators in simple \ca{s} with
  stable rank one,  \emph{Comm. Math. Phys.} \textbf{377} (2020), 37--76.

\bibitem[Tom08]{Tom08ClassificationNuclear}
\bgroup\scshape{}A.~S. Toms\egroup{}, On the classification problem for nuclear
  \ca{s},  \emph{Ann. of Math. (2)} \textbf{167} (2008), 1029--1044.

\bibitem[Vas84]{Vas84BassSR}
\bgroup\scshape{}L.~N. Vaserstein\egroup{}, Bass's first stable range
  condition,  in \emph{Proceedings of the {L}uminy conference on algebraic
  {$K$}-theory ({L}uminy, 1983)}, \textbf{34}, 1984, pp.~319--330.

\bibitem[VW90]{VasWhe90CommutatorsSR1}
\bgroup\scshape{}L.~N. Vaserstein\egroup{} and
  \bgroup\scshape{}E.~Wheland\egroup{}, Commutators and companion matrices over
  rings of stable rank {$1$},  \emph{Linear Algebra Appl.} \textbf{142} (1990),
  263--277.

\bibitem[Vil98]{Vil98SimpleCaPerforation}
\bgroup\scshape{}J.~Villadsen\egroup{}, Simple \ca{s} with perforation,
  \emph{J. Funct. Anal.} \textbf{154} (1998), 110--116.

\bibitem[Wu89]{Wu89OpFactorization}
\bgroup\scshape{}P.~Y. Wu\egroup{}, The operator factorization problems,
  \emph{Linear Algebra Appl.} \textbf{117} (1989), 35--63.

\end{thebibliography}

\providecommand{\bysame}{\leavevmode\hbox to3em{\hrulefill}\thinspace}

\end{document}